\documentclass[a4paper, 12pt]{article}

\usepackage{amsfonts}
\usepackage{amssymb,amsmath}
\usepackage{amsthm}
\usepackage{todonotes}
\usepackage[upright]{fourier}
\usepackage{baskervald} 
\usepackage{dsfont} 
\usepackage{fancyhdr}
\usepackage{graphicx}
\usepackage{float}
\usepackage{wrapfig}
\usepackage{subfigure}
\usepackage{url}

\usepackage{geometry}
\usepackage{verbatim}
\usepackage[english]{babel}
\usepackage[utf8x]{inputenc} 
\usepackage{calligra}
\usepackage[T1]{fontenc}
\usepackage{bbm}
\usepackage{mathrsfs}
\usepackage{color}
\usepackage[color,all]{xy}

\usepackage[hyperfootnotes=false]{hyperref}

\title{Tangent cones to generalised theta divisors and generic injectivity of the theta map}

\author{George {H.} Hitching and Michael Hoff}

\newcommand{\C}{\mathbb{C}}
\newcommand{\PP}{\mathbb{P}}

\newcommand{\cO}{\mathcal{O}}
\newcommand{\tU}{\tilde{U}}
\newcommand{\cG}{\mathcal{G}}
\newcommand{\cU}{{\mathit U}}
\newcommand{\cT}{\mathcal{T}}
\newcommand{\cE}{\mathcal{E}}
\newcommand{\tB}{\tilde{B}}
\newcommand{\bB}{\overline{B}}
\newcommand{\tP}{\tilde{P}}

\newcommand{\End}{\mathrm{End}\,}
\newcommand{\Ker}{\mathrm{Ker}}
\newcommand{\Coker}{\mathrm{Coker}}
\newcommand{\tr}{\mathrm{tr}}
\newcommand{\Image}{\mathrm{Im}\,}
\newcommand{\Iden}{\mathrm{Id}}
\newcommand{\Hom}{\mathrm{Hom}}
\newcommand{\rank}{\mathrm{rk}}
\newcommand{\bp}{p_{1}}

\newcommand{\Oc}{{\mathcal O}_{C}}
\newcommand{\Kc}{K_{C}}
\newcommand{\Tc}{T_{C}}

\newcommand{\cUr}{U'_C (r, r(g-1))}
\newcommand{\SU}{\mathit{SU}}
\newcommand{\cD}{{\mathcal D}}
\newcommand{\Piczero}{\mathrm{Pic}^0(C)}
\newcommand{\isom}{\xrightarrow{\sim}}

\newcommand{\mult}{\mathrm{mult}}
\newcommand{\Pic}{\mathrm{Pic}}
\newcommand{\Gr}{\mathrm{Gr}}

\newtheorem{theorem}{{\textbf Theorem}}[section]
\newtheorem{proposition}[theorem]{{\textbf Proposition}}
\newtheorem{corollary}[theorem]{{\textbf Corollary}}
\newtheorem{lemma}[theorem]{{\textbf Lemma}}

\newtheorem{defn}[theorem]{{\textbf Definition}}
\newtheorem{remit}[theorem]{{\textbf Remark}}
\newenvironment{remark}{\begin{remit}\rm}{\end{remit}}
\newenvironment{definition}{\begin{defn}\rm}{\end{defn}}

\newcommand{\George}[1]{{{\footnotesize{\textcolor{blue}{GEORGE: #1}}}}}

\AtEndDocument{\bigskip{
  \textsc{H\o gskolen i Oslo og Akershus, Postboks 4, St. Olavs plass, 0130 Oslo, Norway.} \par
  \textit{E-mail address}, G.~ H.~Hitching: \texttt{george.hitching@hioa.no} \par
  \textsc{Universit\"at des Saarlandes, Campus E2 4, D-66123 Saarbr\"ucken, Germany} \par  
  \textit{E-mail address}, M.~ Hoff: \texttt{hahn@math.uni-sb.de} \par
}}

\newcommand\blankfootnote[1]{%
  \let\thefootnote\relax\footnotetext{#1}%
  \let\thefootnote\svthefootnote%
}
\let\svfootnote\footnote
\renewcommand\footnote[2][?]{%
  \if\relax#1\relax%
    \blankfootnote{#2}%
  \else%
    \if?#1\svfootnote{#2}\else\svfootnote[#1]{#2}\fi%
  \fi
}

\begin{document}

\maketitle

\footnote[]{2010 \emph{Mathematics subject classification:} 14H60, 14H51, 14M12, 14B10, 13D10, 14C34}
\footnote[]{\emph{Key words and phrases:} Vector bundles, Generalised theta divisors, Brill--Noether theory.}

\begin{abstract} Let $C$ be a Petri general curve of genus $g$ and $E$ a general stable vector bundle of rank $r$ and slope $g-1$ over $C$ with $h^0 (C, E) = r+1$. For $g \ge (2r+2)(2r+1)$, we show how the bundle $E$ can be recovered from the tangent cone to the theta divisor $\Theta_E$ at $\Oc$. We use this to give a constructive proof and a sharpening of Brivio and Verra's theorem that the theta map $\SU_C (r) \dashrightarrow |r \Theta|$ is generically injective for large values of $g$. \end{abstract}

\section{Introduction}

Let $C$ be a nonhyperelliptic curve of genus $g$ and $L \in \Pic^{g-1}(C)$ a line bundle with $h^0 (C, L) = 2$ corresponding to a general double point of the Riemann theta divisor $\Theta$. It is well known that the projectivised tangent cone to $\Theta$ at $L$ is a quadric hypersurface $R_L$ of rank $4$ in the canonical space $|\Kc|^*$, which contains the canonically embedded curve.

Quadrics arising from tangent cones in this way have been much studied: Green \cite{G} showed that the $R_L$ span the space of all quadrics in $|\Kc|^*$ containing $C$; and both Kempf and Schreyer \cite{KS} and Ciliberto and Sernesi \cite{CS} have used the quadrics $R_L$ in various ways to give new proofs of Torelli's theorem.

In another direction: Via the Riemann--Kempf singularity theorem \cite{K73}, one sees that the rulings on $R_L$ cut out the linear series $|L|$ and $|\Kc L^{-1}|$ on the canonical curve. Thus the data of the tangent cone and the canonical curve allows one to reconstruct the line bundle $L$. In this article we study a related construction for vector bundles of higher rank.

Let $V \to C$ be a semistable vector bundle of rank $r$ and integral slope $h$. We consider the set 
\begin{equation} \label{ThetaVsupport}
\left\lbrace M \in \Pic^{g-1-h}(C)\ \colon\ h^0 (C, V \otimes M) \ge 1 \right\rbrace .
\end{equation}
It is by now well known that for general $V$, this is the support of a divisor $\Theta_V$ algebraically equivalent to a translate of $r \cdot \Theta$. If $V$ has trivial determinant, then in fact $\Theta_V$, when it exists, belongs to $|r \Theta|$. 
%

For general $V$, the projectivised tangent cone $\cT_M (\Theta_V)$ to $\Theta_V$ at a point $M$ of multiplicity $r+1$ is a determinantal hypersurface of degree $r+1$ in $|\Kc|^*$ (see for example Casalaina Martin--Teixidor i Bigas \cite{CT}). Our first main result (\S \ref{reconstruction}) is a construction which from $\cT_M (\Theta_V)$ recovers the bundle $V \otimes M$, up to the involution $V \otimes M \mapsto \Kc \otimes M^{-1} \otimes V^*$. This is valid whenever $V \otimes M$ and $\Kc \otimes M^{-1} \otimes V^*$ are globally generated.

We apply this construction to give an improvement of a result of Brivio and Verra \cite{BV}. To describe this application, we need to recall some more objects. Write $\SU_C(r)$ for the moduli space of semistable bundles of rank $r$ and trivial determinant over $C$. The association $V \mapsto \Theta_V$ defines a map
\begin{equation} 
\cD \colon \SU_C (r) \ \dashrightarrow \ | r\Theta | \ = \ \PP^{r^g - 1} , \label{thetamap}
\end{equation}
called the \textsl{theta map}. Drezet and Narasimhan \cite{DN} showed that the line bundle associated to the theta map is the ample generator of the Picard group of $\SU_C(r)$. Moreover, the indeterminacy locus of $\cD$ consists of those bundles $V \in \SU_C(r)$ for which (\ref{ThetaVsupport}) is the whole Picard variety. This has been much studied; see for example Pauly \cite{Pau10}, Popa \cite{Popa99} and Raynaud \cite{Ray82}.

Brivio and Verra \cite{BV} showed that $\cD$ is generically injective for a general curve of genus $g \ge \binom{3r}{r} - 2r - 1$, partially answering a conjecture of Beauville \cite[\S 6]{B06}. 
 We apply the aforementioned construction to give the following sharpening of Brivio and Verra's result:

\begin{theorem} For $r \ge 2$ and $C$ a Petri general curve of genus $g \ge (2r+2)(2r+1)$, the theta map (\ref{thetamap}) is generically injective. \label{thetageninj} \end{theorem}

In addition to giving the statement for several new values of $g$ when $r \ge 3$ (our lower bound for $g$ depends quadratically on $r$ rather than exponentially), our proof is constructive, based on the method mentioned above for explicitly recovering the bundle $V$ from the tangent cone to the theta divisor at a point of multiplicity $r+1$. This gives a new example, in the context of vector bundles, of the principle apparent in \cite{KS} and \cite{CS} that the geometry of a theta divisor at a sufficiently singular point can encode essentially all the information of the bundle and/or the curve.

Our method works for $r = 2$, but in this case much more is already known: Narasimhan and Ramanan \cite{NR69} showed, for $g = 2$ and $r = 2$, that $\cD$ is an isomorphism $\SU_C(2) \isom \PP^3$, and van Geemen and Izadi \cite{vGI} generalised this statement to nonhyperelliptic curves of higher genus. Note that our proof of Theorem \ref{thetageninj} is not valid for hyperelliptic curves (see Remark \ref{failure}).

Here is a more detailed summary of the article. In \S \ref{genThetaDiv}, we study semistable bundles $E$ of slope $g-1$ for which the \textsl{Petri trace map}
\[ \bar{\mu} \colon H^0 (C, E) \otimes H^0 (C, \Kc \otimes E^*) \to H^0 (C, \Kc) \]
is injective. A bundle $E$ with this property will be called \textsl{Petri trace injective}. We prove that for large enough genus, the theta divisor of a generic $V \in \SU_C(r)$ contains a point $M$ of multiplicity $r+1$ such that $V \otimes M$ and $\Kc \otimes M^{-1} \otimes V^*$ are Petri trace injective and globally generated.

Suppose now that $E$ is a vector bundle of slope $g-1$ with $h^0 (C, E) \ge 1$. If $\Theta_E$ is defined and $\mult_{\Oc} ( \Theta_E ) = h^0 (C, E)$, then the tangent cone to $\Theta_E$ at $\Oc$ is a determinantal hypersurface in $|\Kc|^* = \PP^{g-1}$ containing the canonical embedding of $C$. We prove (Proposition \ref{recovery} and Corollary \ref{transpose}) that if $C$ is a general curve of genus $g \ge (2r+2)(2r+1)$, and $E$ a globally generated Petri trace injective bundle of rank $r$ and slope $g-1$ with $h^0 (C, E) = r+1$, then the bundle $E$ can be reconstructed up to the involution $E \mapsto \Kc \otimes E^*$ from a certain determinantal representation of the tangent cone to $\Theta_E$ at $\Oc$. By a classical result of Frobenius (whose proof we sketch in Proposition \ref{complexes}), any two such representations are equivalent up to transpose. The generic injectivity of the theta map for a Petri general curve (Theorem \ref{Thetageninj}) can then be deduced by combining these facts and the statement in \S \ref{genThetaDiv} that the theta divisor of a general $V \in \SU_C(r)$ contains a suitable point of multiplicity $r+1$.

We assume throughout that the ground field is $\C$. The reconstruction of $E$ from its tangent cone in \S \ref{reconstruction}
is valid for an algebraically closed field of characteristic zero or $p > 0$ not dividing $r+1$.

\paragraph*{Acknowledgements} 
The first named author gratefully acknowledges financial support from H\o gskolen i Oslo Oppstarts- og Akkvisisjonsmidler 2015. He also thanks the Universit\"at des Saarlandes for hospitality. The second named author was partially supported by the DFG-grant SPP 1489 Schr{.} 307/5-2. The authors would like to thank Ciro Ciliberto, Peter Newstead, Christian Pauly, Frank-Olaf Schreyer and Eduardo Sernesi for helpful discussions and comments on this work. 

\section{Singularities of theta divisors of vector bundles} \label{genThetaDiv}

\subsection{Petri trace injective bundles}


Let $C$ be a projective smooth curve of genus $g \ge 2$. Let $V \to C$ be a stable vector bundle of rank $r \ge 2$ and integral slope $h$, and consider the locus
\begin{equation}
\left\lbrace M \in \Pic^{g-1-h}(C) : h^0 (C, V \otimes M) \ge 1 \right\rbrace .
\label{thetadiv}
\end{equation}
If this is not the whole of $\Pic^{g-1-h}(C)$, then it is the support of the theta divisor $\Theta_V$. 

The theta divisor of a vector bundle is a special case of a \textsl{twisted Brill--Noether locus}
\begin{equation} B_{1, g-1-h}^n(V) \ := \ \left\lbrace M \in \Pic^{g-1-h}(C) : h^0 (C, V \otimes M) \ge n \right\rbrace . \label{tBNl} \end{equation}
The following is central in the study of these loci (see for example Teixidor i Bigas \cite[\S 1]{TiB}): For $E \to C$ a stable vector bundle, we consider the \textsl{Petri trace map}:
\begin{equation}
\bar{\mu} \colon H^0(C, E)\otimes H^0(C,\Kc\otimes E^*) \stackrel{\mu}{\to} H^0(C,\Kc\otimes \End E) \stackrel{\tr}{\to} H^0(C,\Kc) .
\label{barmu} \end{equation}
Then for $E = V \otimes M$ and $M\in B^n_{1, g-1-h}(V)\backslash B^{n+1}_{1, g-1-h}(V)$, the Zariski tangent space to $B^n_{1, g-1-h}(V)$ at $M$ is exactly $\Image \left( \bar{\mu} \right)^{\perp}$. 
 This motivates a definition:

\begin{definition} Suppose $E \to C$ is a vector bundle with $h^0 (C, E) = n \ge 1$. If the map $\mu$ above is injective, we will say that $E$ is \textsl{Petri injective}. If the composed map $\bar{\mu}$ is injective, we will say that $E$ is \textsl{Petri trace injective}.

\end{definition}

\begin{remark} \label{ptiproperties} \hspace{2em} \begin{enumerate}
\item [(1)] Clearly, a Petri trace injective bundle is Petri injective. For line bundles, the two notions coincide.
\item [(2)] Suppose $V \in \mathit{U}_C(r, d)$ where $\mathit{U}_C(r, d)$ is the moduli space of semistable rank $r$ vector bundles of degree $d$. 
If $E = V \otimes M$ is Petri trace injective for $M \in \Pic^e (C)$, then $B^n_{1,e}(V)$ is smooth at $M$ and of the expected dimension
\[ h^1 (C, \Oc) - h^0 (C, V \otimes M) \cdot h^1 (C, V \otimes M) . \]
\item [(3)] We will also need to refer to the usual generalised Brill--Noether locus 
\[ B^n_{r, d} \ = \ \left\lbrace E \in \mathit{U}_C (r, d) : h^0 (C, E) \ge n \right\rbrace . \]
If $E$ is Petri injective then this is smooth and of the expected dimension
\[ h^1 (C, \End E) - h^0 (C, E) \cdot h^1 (C, E) \]
at $E$. See for example Grzegorczyk and Teixidor i Bigas \cite[\S 2]{GTiB}. 
\item [(4)] Petri injectivity and Petri trace injectivity are open conditions on families of bundles ${\mathcal E} \to C \times B$ with $h^0 (C, {\mathcal E}_b)$ constant. Later, we will discuss the sense in which these properties are ``open'' when $h^0 (C, {\mathcal E}_b)$ may vary.

\end{enumerate} \end{remark}

\noindent We will also need the notion of a Petri general curve:

\begin{definition} A curve $C$ is called \textsl{Petri general} if every line bundle on $C$ is Petri injective. \end{definition}

By \cite{Gie82}, the locus of curves which are not Petri general is a proper subset of the moduli space $M_g$ of curves of genus $g$, the so called \textsl{Gieseker--Petri locus}. The hyperelliptic locus is contained in the Gieseker--Petri locus. Apart from this, in general not much is known about the components of the Gieseker--Petri locus and their dimensions. For an overview of known results, we refer to \cite{TiB88}, \cite{Far05} and \cite{BS11} and the references cited there. 

\begin{proposition} Suppose $V$ is a stable bundle of rank $r$ and integral slope $h$. Suppose $M_0 \in \Pic^{g-1-h}(C)$ satisfies $h^0 (C, V \otimes M_0) \ge 1$, and furthermore that $V \otimes M_0$ is Petri trace injective. Then the theta divisor $\Theta_V \subset \Pic^{g-1-h}(C)$ is defined. Furthermore, we have equality $\mult_{M_0} \Theta_V = h^0 (C, V \otimes M_0)$. \label{ptithetagood} \end{proposition}

\begin{proof} Write $E := V \otimes M_0$. It is well known that via Serre duality, $\bar{\mu}$ is dual to the cup product map
\[ \cup \colon H^1 (C, \Oc) \to \Hom \left( H^0 (C, E) , H^1 (C, E) \right) . \]
By hypothesis, therefore, $\cup$ is surjective. Since $E$ has Euler characteristic zero, $h^0 (C, E) = h^1 (C, E)$. Hence there exists $b \in H^1 (C, \Oc)$ such that $\cdot \cup b \colon H^0 (C, E) \to H^1 (C, E)$ is injective. The tangent vector $b$ induces a deformation of $M_0$ and hence of $E$, which does not preserve any nonzero section of $E$. Therefore, the locus
\[ \left\lbrace M \in \Pic^{g-1-h}(C) : h^0 (C, V \otimes M) \ge 1 \right\rbrace \]
is a proper sublocus of $\Pic^{g-1-h} (C)$, so $\Theta_V$ is defined. Now we can apply Casalaina--Martin and Teixidor i Bigas \cite[Proposition 4.1]{CT}, to obtain the desired equality $\mult_{M_0} \Theta_V = h^0 (C, V \otimes M_0)$. \end{proof}

\subsection{Existence of good singular points}

In this section, we study global generatedness and Petri trace injectivity of the bundles $V \otimes M$ for $M \in B^{r+1}_{1, g-1}(V)$ for general $C$ and $V$. The main result of this section is:

\begin{theorem} Suppose $C$ is a Petri general curve of genus $g \ge (2r+2)(2r+1)$ and $V \in \SU_C(r)$ a general bundle. Then there exists $M \in \Theta_V$ such that $h^0 (C, V \otimes M) = r+1$, and both $V \otimes M$ and $\Kc \otimes M^{-1} \otimes V^*$ are globally generated and Petri trace injective. \label{MExists} \end{theorem}

The proof of this theorem has several ingredients. We begin by constructing a stable bundle $E_0$ with some of the properties we are interested in. Let $F$ be a semistable bundle of rank $r-1$ and degree $(r-1)(g-1) - 1$, and let $N$ be a line bundle of degree $g$. 
%

\begin{lemma} A general extension $0 \to F \to E \to N \to 0$ is a stable vector bundle. \label{EStable} \end{lemma}

\begin{proof} Any subbundle $G$ of $E$ fits into a diagram
\[ \xymatrix{ 0 \ar[r] & G_1 \ar[r] \ar[d]_{\iota_1} & G \ar[r] \ar[d] & N(-D) \ar[d]_{\iota_2} \ar[r] \ar[d] & 0 \\
 0 \ar[r] & F \ar[r] & E \ar[r] & N \ar[r] & 0 } \]
where $D$ is an effective divisor on $C$. If $\iota_2 = 0$, then $\mu(G) = \mu(G_1) \le \mu (F) < \mu(E)$. Suppose $\iota_2 \ne 0$, and write $s := \rank ( G_1 )$. If $s \ne 0$, the semistability of $F$ implies that
\[ \deg (G_1) \ \le \ s(g-1) - \frac{s}{r-1} , \]
so in fact $\deg (G_1) \le s(g-1) - 1$. As $\deg (N) = g$, we have $\deg (G) \le (s+1)(g-1)$. Thus we need only exclude the case where $\deg (G_1) = s(g-1) - 1$ and $D = 0$, so $\iota_2 = \Iden_N$. In this case, the existence of the above diagram is equivalent to $[E] = (\iota_1)_* [G]$ for some extension $G$, that is, $[E] \in \Image \left( ( \iota_1 )_* \right)$. 
 It therefore suffices to check that
\[ (\iota_1)_* \colon H^1 ( C, \Hom(N, G_1)) \ \to \ H^1 ( C, \Hom(N, F)) \]
is not surjective. This follows from the fact, easily shown by a Riemann--Roch calculation, that $h^1 ( C, \Hom(N, F/G_1)) > 0$.


If $s = 0$, then we need to exclude the lifting of $G = N(-p)$ for all $p \in C$, that is,
\[ [E] \ \not\in \ \bigcup_{p \in C} \left( \Ker \left( H^1 (C, \Hom( N, F)) \ \to \ H^1 (C, \Hom (N(-p), F)) \right) \right) . \]
A dimension count shows that this locus is not dense in $H^1 (C, \Hom (N, F))$.
\end{proof} 

\begin{lemma} \label{CoboundarySurjective} Suppose $h^0 (C, N) \ge h^1 (C, F)$. Then for a general extension $0 \to F \to E \to N \to 0$, the coboundary map is surjective. \end{lemma}

\begin{proof} Clearly it suffices to exhibit one extension $E_0$ with the required property. We write $n := h^1 (C, F)$ for brevity.

Let $0 \to F \to \tilde{F} \to \tau \to 0$ be an elementary transformation with $\deg (\tau) = n$ and such that the image of $\Gamma(C, \tau)$ generates $H^1 (C, F)$. We may assume that $\tau$ is supported along $n$ general points $p_1 , \ldots , p_n$ of $C$ which are not base points of $|N|$. Then $\tau_{p_i}$ is generated by an element
\[ \phi_i \ \in \ \left( \frac{F(p_i)}{F}\right)_{p_i} \]
defined up to nonzero scalar multiple. We write $[ \phi_i ]$ for the class in $H^1 (C, F)$ defined by $\phi_i$. 

Now $h^0 (C, N) \ge n$ and the image of $C$ is nondegenerate in $|N|^*$. 
 As the $p_i$ can be assumed to be general, they impose independent conditions on sections of $N$. We choose sections $s_1, \ldots , s_n \in H^0 ( C, N )$ such that $s_i ( p_i ) \neq 0$ but $s_i (p_j) = 0$ for $j \ne i$. 
 For $1 \le i \le n$, let $\eta_i$ be a local section of $N^{-1}$ such that $\eta_i (s_i (p_i) ) = 1$.

Let $0 \to F \to E_0 \to N \to 0$ be the extension with class $[E_0]$ defined by the image of
\[ \left( \eta_1 \otimes \phi_1 , \ldots , \eta_n \otimes \phi_n \right) \]
by the coboundary map $\Gamma (C, N^{-1} \otimes \tau ) \to H^1 (C, N^{-1} \otimes F)$. Then $s_i \cup [E_0] = [ \phi_i ]$ for $1 \le i \le n$. Hence the image of $\cdot \cup [E_0]$ spans $H^1 (C, F)$. \end{proof}

We now make further assumptions on $F$ and $N$. If $r = 2$, then $g \ge (2r+2)(2r+1) = 30$. Hence by the Brill--Noether theory of line bundles on $C$, we may choose a line bundle $F$ of degree $g-2$ with $h^0 (C, F) = 2$ and $|F|$ base point free. 
 If $r \ge 3$: Since $g \ge 3$, we have $(r-1)(g-1) - 1 \ge r$. Therefore, by \cite[Theorem 5.1]{BBN2} we may choose a semistable bundle $F$ of rank $r-1$ and degree $(r-1)(g-1) - 1$ which is generated and satisfies $h^0 (C, F) = r$, so $h^1 (C, F) = r + 1$. 

Furthermore, again by Brill-Noether theory, since $g \ge (2r+2)(2r+1)$ we may choose $N \in \Pic^g (C)$ such that $h^0 (C, N) = 2r+2$ and $|N|$ is base point free. By Lemma \ref{CoboundarySurjective}, we may choose an $(r+1)$-dimensional subspace $\Pi \subset H^0 (C, E)$ lifting from $H^0 (C, N)$. 

\begin{proposition} \label{RestrictedPetriTraceMap} Let $F$, $N$ and $\Pi$ be as above, and let $0 \to F \to E \to N \to 0$ be a general extension. Then the restricted Petri trace map $\Pi \otimes H^0 ( C, \Kc \otimes E^* ) \ \to \ H^0 ( C, \Kc )$ is injective. \end{proposition}

\begin{proof} Choose a basis $\sigma_1 , \ldots , \sigma_{r+1}$ for $\Pi$. For each $i$, write $\widetilde{\sigma_i}$ for the image of $\sigma_i$ in $H^0 (C, N)$.

By Lemma \ref{CoboundarySurjective}, there is an isomorphism $H^1 (C, E) \ \xrightarrow{\sim} \ H^1 (C, N)$. Hence, by Serre duality, the injection $\Kc \otimes N^{-1} \hookrightarrow \Kc \otimes E^*$ induces an isomorphism on global sections. Choose a basis $\tau_1 , \ldots , \tau_{2r+1}$ for $H^0 (C, \Kc \otimes E^*)$. For each $j$, write $\widetilde{\tau_j}$ for the preimage of $\tau_j$ by the aforementioned isomorphism.

For each $i$ and $j$ we have a commutative diagram
\[ \xymatrix{ E^* \ar[r]^{{^t\sigma_i}} & \Oc \ar[r]^-{\tau_j} \ar[dr]_{\widetilde{\tau_j}} & \Kc \otimes E^* \\
 N^{-1} \ar[u] \ar[ur]_{\widetilde{\sigma_i}} & & \Kc \otimes N^{-1} \ar[u] } \]
where the top row defines the twisted endomorphism
\[ \mu \left( \sigma_i \otimes \tau_j \right) \ \in \ H^0 (C, \Kc \otimes \End E^* ) \ = \ H^0 (C, \Kc \otimes \End E) . \]
Clearly this has rank one. As it factorises via $\Kc \otimes N^{-1}$, at a general point of $C$ the eigenspace corresponding to the single nonzero eigenvalue is identified with the fibre of $N^{-1}$ in $E^*$. Hence the Petri trace $\bar{\mu}( \sigma_i \otimes \tau_j)$ may be identified with the restriction to $N^{-1}$. By the diagram, we may identify this restriction with
\[ \mu_N \left( \widetilde{\sigma_i} \otimes \widetilde{\tau_j} \right) \ \in \ H^0 (C, \Kc) , \]
where $\mu_N$ is the Petri map of the line bundle $N$. Since $C$ is Petri, $\mu_N$ is injective. Hence the elements $\bar{\mu}(\sigma_i \otimes \tau_j) = \mu_N ( \widetilde{\sigma_i} \otimes \widetilde{\tau_j} )$ are independent in $H^0 (C, \Kc)$. This proves the statement. \end{proof}

Before proceeding, we need to recall some background on coherent systems (see \cite[\S 2]{BBN} for an overview and \cite{BGMN} for more detail): We recall that a \textsl{coherent system of type $(r, d, k)$} is a pair $(W, \Pi)$ where $W$ is a vector bundle of rank $r$ and degree $d$ over $C$, and $\Pi \subseteq H^0 (C, W)$ is a subspace of dimension $k$. There is a stability condition for coherent systems depending on a real parameter $\alpha$, and a moduli space $G(\alpha; r, d, k)$ for equivalence classes of $\alpha$-semistable coherent systems of type $(r, d, k)$. If $k \ge r$, then by \cite[Proposition 4.6]{BGMN} there exists $\alpha_L \in \mathbb{R}$ such that $G(\alpha; r, d, k)$ is independent of $\alpha$ for $\alpha > \alpha_L$. This ``terminal'' moduli space is denoted $G_L$. Moreover, the locus
\[ U(r, d, k) \ := \ \{ (W, \Pi) \in G_L : \hbox{ $W$ is a stable vector bundle} \} \]
is an open subset of $G_L$. For us, $d = r(g-1)$ and $k = r+1$. To ease notation, we write $U := U(r, r(g-1), r+1)$.

Let now $N_1$ be a line bundle of degree $g$ with $h^0 (C, N_1) \ge r+2$. Let $F$ be as above, and let $0 \to F \to E \to N_1 \to 0$ be a general extension.

\begin{lemma} For a general subspace $\Pi \subset H^0 (C, E)$ of dimension $r+1$, the coherent system $(E, \Pi)$ defines an element of $U$. \end{lemma}

\begin{proof} Recall the bundle $E_0$ defined in Lemma \ref{CoboundarySurjective}, which clearly is generically generated. Let us describe the subsheaf $E_0'$ generated by $H^0 (C, E_0)$.

Write $p_1 + \cdots + p_{r+1} =: D$. Clearly $s \cup [E_0] = 0$ for any $s \in H^0 (C, N_1(-D))$. Since the $p_i$ are general points,
\[ h^0 (C, N_1(-D)) \ = \ h^0 (C, N_1) - (r+1) \ = \ \dim \left( \Ker (\cdot \cup [E_0] \colon H^0 (C, N_1) \to H^1 (C, F) \right) . \]
Therefore, the image of $H^0 (C, E_0)$ in $H^0 (C, N_1)$ is exactly $H^0 (C, N_1(-D))$. As the subbundle $F \subset E_0$ is globally generated, $E_0'$ is an extension $0 \to F \to E_0' \to N_1(-D) \to 0$. Dualising and taking global sections, we obtain
\[ 0 \to H^0 ( C, N_1^{-1}(D)) \to H^0 (C, (E_0')^*) \to H^0 (C, F^*) \to \cdots \]
Since both $N_1^{-1}(D)$ and $F^*$ are semistable of negative degree, 
 $h^0 ( C, N_1^{-1}(-D)) = h^0 (C, F^*) = 0$, so $h^0 (C, (E_0')^*) = 0$.

Now let $\Pi_1 \subset H^0 (C, E_0)$ be any subspace of dimension $r+1$ generically generating $E_0$. Since $h^0 (C, (E_0')^*) = 0$, by \cite[Theorem 3.1 (3)]{BBN} the coherent system $(E_0, \Pi_1)$ defines a point of $G_L$. Since generic generatedness and vanishing of $h^0 (C, (E')^*)$ are open conditions on families of bundles with a fixed number of sections, the same is true for a generic $(E, \Pi)$ where $E$ is an extension $0 \to F \to E \to N_1 \to 0$. By Lemma \ref{EStable}, in fact $(E, \Pi)$ belongs to $U$. \end{proof}
 
\begin{lemma} For generic $E$ represented in $U$, we have $h^0 (C, E) = h^0 (C, \Kc \otimes E^*) = r+1$. \label{rPlusOne} \end{lemma}

\begin{proof} Since $C$ is Petri general, $B^{r+2}_{1, g}$ is irreducible in $\Pic^g (C)$. Thus there exists an irreducible family parametrising extensions of the form $0 \to F \to E \to N_1 \to 0$ where $F$ is as above and $N_1$ ranges over $B^{r+2}_{1, g}$. This contains the extension $E_0$ constructed above. By Lemma \ref{CoboundarySurjective}, a general element $E_1$ of the family satisfies $h^0 (C, E_1) = r+1$. By semicontinuity, the same is true for general $E$ represented in $U$. \end{proof}

Now by \cite[Theorem 3.1 (4) and Remark 6.2]{BBN}, the locus $U$ is irreducible. Write $B$ for the component of $B^{r+1}_{r, r(g-1)}$ containing the image of $U$, and $B'$ for the sublocus $\{ E \in B : h^0 (C, E) = r+1\}$. Set $U' := U \times_B B'$; clearly $U' \cong B'$.

Let $\tB \to B$ be an \'etale cover such that there is a Poincar\'e bundle $\cE \to \tB \times C$. In a natural way we obtain a commutative cube
\[ \xymatrix{ \tU' \ar[r] \ar[drr] \ar[d]^\wr & \tU \ar[d] \ar[drr] & & \\
 \tB' \ar[r] \ar[drr] & \tB \ar[drr] & U' \ar[r] \ar[d]^\wr & U \ar[d] \\
& & B' \ar[r] & B } \]
where all faces are fibre product diagrams. 
 By a standard construction, we can find a complex of bundles $\alpha \colon K^0 \to K^1$ over $\tB$ such that $\Ker (\alpha_b ) \cong H^0 (C, \cE_b)$ for each $b \in \tB$. Following \cite[Chapter IV]{ACGH}, we consider the Grassmann bundle $\Gr (r+1, K^0 )$ over $\tB$ and the sublocus
\[ \cG \ := \ \{ \Lambda \in \Gr (r+1, K^0) : \alpha|_{\Lambda} = 0 \} . \]
Write $\cG_1 := \cG \times_{\tB} \tU$. The fibre of $\cG_1$ over $(E_b , \Pi) \in \tU$ is then $\Gr (r+1, H^0 (C, \Kc \otimes E_b^*))$.

Now let $E_0$ be a bundle as constructed in Lemma \ref{CoboundarySurjective} with $h^0 (C, E_0) = 2r+2$, and let $\Pi_0$ be a generic choice of $(r+1)$-dimensional subspace of $H^0 (C, E_0)$. We may assume $\tU$ is irreducible since $U$ is. Since $\tU \to U$ is \'etale, by Lemma \ref{RestrictedPetriTraceMap} in fact $U$ is also smooth at $(E_0, \Pi_0)$ (cf{.} \cite[Proposition 3.10]{BGMN}). Therefore, we may choose a one-parameter family $\{ (E_t, \Pi_t) : t \in T \}$ in $\tU$ such that $(E_{t_0}, \Pi_{t_0}) = (E_0, \Pi_0)$ while $(E_t, \Pi_t)$ belongs to $\tU'$ for generic $t \in T$. Since the bundles have Euler characteristic zero, for generic $t \in T$ there is exactly one choice of $\Lambda \in \cG_1|_{(E_t, \Pi_t)}$. Thus we obtain a section $T \backslash \{ 0 \} \to \cG_1$. As $\dim T = 1$, this section can be extended uniquely to $0$. We obtain thus a triple $(E_0, \Pi_0, \Lambda_0)$ where $\Lambda \subset H^0 (C, \Kc \otimes E_0^*)$ has dimension $r+1$. By Lemma \ref{RestrictedPetriTraceMap}, this triple is Petri trace injective. Hence
\[ (E_t, \Pi_t, \Lambda_t) \ = \ (E, H^0 (C, E), H^0 (C, \Kc \otimes E^*) \]
is Petri trace injective for generic $t \in T$. 
 Thus a general bundle $E$ represented in $\tU'$ is Petri trace injective.

Furthermore, by \cite[Theorem 3.1 (4)]{BBN}, a general $(E, \Pi) \in U$ is globally generated (not just generically). Thus we obtain:

\begin{proposition} A general element $E$ of the irreducible component $B \subseteq B^{r+1}_{r, r(g-1)}$ is Petri trace injective and globally generated with $h^0 (C, E) = r+1$. \label{BgoodExists} \end{proposition}

Now we can prove the theorem:

\begin{proof}[Proof of Theorem \ref{MExists}] Consider the map $a \colon \SU_C(r) \times \Pic^{g-1}(C) \to \cU_C(r, r(g-1))$ given by $(V, M) \mapsto V \otimes M$. This is an \'etale cover of degree $r^{2g}$. We write $\bB$ for the inverse image $a^{-1} (B)$. 
Since $a$ is \'etale, we have $T_{(V, M)} \bB \ \cong \ T_{V \otimes M} B$ for each $(V, M) \in B$. In particular,
\begin{equation} \dim \bB \ = \ \dim B \ = \ \dim \cU_C(r, r(g-1)) - (r+1)^2 . \label{dimtB} \end{equation}

We write $p$ for the projection $\SU_C (r) \times \Pic^{g-1}(C) \to \SU_C(r)$, and $\bp$ for the restriction $p|_{\bB} \colon \bB \to \SU_C(r)$.\\
\\
\noindent \textbf{Claim:} $\bp$ is dominant.

To see this: For $(V, M) \in \bB$, the locus $\bp^{-1} (V)$ is identified with an open subset of the twisted Brill--Noether locus
\[ B^{r+1}_{1, g-1}(V) \ = \ \left\lbrace M \in \Pic^{g-1}(C) : h^0 (C, V \otimes M) \ge r+1 \right\rbrace \ \subseteq \ \Pic^{g-1}(C). \]
Moreover, for each such $(V, M)$, we have 
\[ \dim_M \left( \bp^{-1}(V) \right) \ = \ \dim \left( T_M \left( B^{r+1}_{1, g-1}(V) \right) \right) \ = \ \dim \Image ( \bar{\mu} )^{\perp}. \]
Since $V \otimes M$ is Petri trace injective, this dimension is $g - (r+1)^2$. By semicontinuity, a general fibre of $\bp$ has dimension at most $g - (r+1)^2$. Therefore, in view of (\ref{dimtB}), the image of $\bp$ has dimension at least
\[ \left( \dim \cU_C(r, r(g-1)) - (r+1)^2 \right) - \left( g - (r+1)^2 \right) \ = \ \dim \cU_C(r, r(g-1)) - g \ = \ \dim \SU_C(r) . \]
As $\SU_C(r)$ is irreducible, the claim follows.

Now we can finish the proof: Let $V \in \SU_C(r)$ be general. By the claim, we can find $(V, M) \in \tP$ such that $h^0 (C, V \otimes M) = r+1$ and $V \otimes M$ is globally generated and Petri trace injective. By Proposition \ref{ptithetagood}, the theta divisor $\Theta_V$ exists and satisfies $\mult_M \Theta_V = h^0 (C, V \otimes M) = r+1$. Lastly, by considering a suitable sum of line bundles, 
 one sees that the involution $E \mapsto \Kc \otimes E^*$ preserves the component $\bB$. Since a general element of $\bB$ is globally generated, in general both $V \otimes M$ and $\Kc \otimes M^{-1} \otimes V^*$ are globally generated. \end{proof}

\section{Reconstruction of bundles from tangent cones to theta divisors}
\label{reconstructionTangentCone}

\subsection{Tangent cones}

Let $Y$ be a normal variety and $Z \subset Y$ a divisor. 
Let $p$ be a smooth point of $Y$ which is a point of multiplicity $n \ge 1$ of $Z$. A local equation $f$ for $Z$ near $p$ has the form $f_n + f_{n+1} + \cdots$, where the $f_i$ are homogeneous polynomials of degree $i$ in local coordinates centred at $p$. The projectivised tangent cone $\cT_p (Z)$ to $Z$ at $p$ is the hypersurface in $\PP T_p Y$ defined by the first nonzero component $f_n$ of $f$. (For a more intrinsic description, see \cite[Chapter II.1]{ACGH}.)

Now let $C$ be a curve of genus $g\ge (r+1)^2$. Let $E$ be a Petri trace injective bundle of rank $r$ and degree $r(g-1)$, with $h^0 (C, E) = r+1$. By Proposition \ref{ptithetagood} (with $h = g-1$), the theta divisor $\Theta_E$ is defined and contains the origin $\Oc$ of $\Piczero$ with multiplicity $h^0(C,E) = r+1$.

By \cite[Theorem 3.4 and Remark 3.8]{CT} (see also Kempf \cite{K73}), the tangent cone $\cT_{\Oc} (\Theta_E)$ to $\Theta_E$ at $\Oc$ is given by the determinant of an $(r+1)\times (r+1)$ matrix $\Lambda = \begin{pmatrix} l_{ij} \end{pmatrix}$ of linear forms $l_{ij}$ on $H^1 (C, \Oc)$, which is related to the multiplication map $\bar{\mu}$ as follows: In appropriate bases $(s_i)$ and $(t_j)$ of $H^0(C,E)$ and $H^0(C,\Kc\otimes E^*)$ respectively, $\Lambda$ is given by 
$$
\left( l_{ij} \right) \ = \ \left( \bar{\mu}(s_i\otimes t_j) \right) .
$$
Hence, via Serre duality, $\Lambda$ coincides with the cup product map
$$
\cup \colon H^0(C,E) \otimes H^1(C,\Oc) \to H^1(C,E) .
$$
Thus the matrix $\Lambda= \left( l_{ij} \right)$ is a matrix of linear forms on the canonical space $|K_C|^*$.

In the following two subsections, we will show on the one hand that one can recover the bundle $E$ from the determinantal representation of the tangent cone $\cT_{\Oc}(\Theta_E)$ given by the matrix $\Lambda$. On the other hand, up to changing bases in $H^0(C, E)$ and $H^1 (C, E)$ there are only two determinantal representations of the tangent cone, namely $\Lambda$ or $\Lambda^t$. Thus the tangent cone determines $E$ up to an involution.

We will denote by $\varphi$ the canonical embedding $C \hookrightarrow |\Kc|^*$.

\subsection{Reconstruction of the bundle from the tangent cone} \label{reconstruction}

As above, let $\Lambda=\left(l_{ij}\right)$ be the determinantal representation of the tangent cone given by the cup product mapping. 
We identify the source of $\Lambda$ with $H^0(C,E)$ and the target with $H^1(C,E)$:
$$
H^0(C,E)\otimes \cO_{\PP}(-1) \stackrel{\Lambda}{\longrightarrow} H^1(C,E) \otimes \cO_{\PP}.
$$

We recall that the Serre duality isomorphism sends $b \in H^1 (C, E)$ to the linear form
\[ \cdot \cup b \colon H^0 (C, \Kc \otimes E^* ) \ \to \ H^1 (C, \Kc ) \ = \ \C . \]

In the following proofs, we will use principal parts in order to represent cohomology classes of certain bundles. We refer to \cite{K1} or \cite[\S 3.2]{Pau} for the necessary background. See also Kempf and Schreyer \cite{KS}.

\begin{lemma} Suppose that $h^0 (C, E) = r+1$ and $E$ and $\Kc \otimes E^*$ are globally generated. Then the rank of $\Lambda|_C = \varphi^* \Lambda$ is equal to $r=\rank \, E$ at all points of $C$. In particular, the canonical curve is contained in $\cT_{\Oc} ( \Theta_E )$. \label{rankr} \end{lemma}

\begin{proof} For each $p \in C$, write $\beta_p$ for a principal part with a simple pole supported at $p$. Then (see \cite{KS}) the cohomology class $\left[ \beta_p \right]$ is identified with the image of $p$ by $\varphi$. Therefore, at $p \in C$, the pullback $\Lambda|_C$ is identified with the cup product map
\[ \left[ \beta_p \right] \otimes s \ \mapsto \ \left[ \beta_p \right] \cup s . \]
The kernel of $\left[ \beta_p \right] \cup \cdot$ contains the subspace $H^0 (C, E(-p))$, which is one-dimensional since $E$ is globally generated and $h^0 (C, E) = r+1$. If $\Ker \left( \left[ \beta_p \right] \cup \cdot \right)$ has dimension greater than 1, then there is a section $s' \in H^0 (C, E)$ not vanishing at $p$ such that
\[ \left[ \beta_p \cdot s' \right] \ = \ \left[ \beta_p \right] \cup s' \ = \ 0 \ \in \ H^1 (C, E) . \]
By Serre duality, this means that
\[ \left[ \beta_p \cdot \langle s'(p) , t(p) \rangle \right] \]
is zero in $H^1 (C, \Kc)$ for all $t \in H^0 (C, \Kc \otimes E^*)$. Hence the values at $p$ of all global sections of $\Kc \otimes E^*$ belong to the hyperplane in $(\Kc \otimes E^*)|_p$ defined by contraction with the nonzero vector $s'(p) \in E|_p$. Thus $\Kc \otimes E^*$ is not globally generated, contrary to our hypothesis. \end{proof}

\begin{remark} Casalaina-Martin and Teixidor i Bigas in \cite[\S 6]{CT} prove more generally that if $E$ is a general vector bundle with $h^0 (C, E) > kr$, then the $k$th secant variety of the canonical image $\varphi(C)$ of $C$ is contained in $\cT_{\Oc}( \Theta_E)$. \end{remark}

\begin{proposition}
\label{recovery}
Let $E$ be a vector bundle with $h^0 (C, E) = r+1$, such that both $E$ and $\Kc \otimes E^*$ are globally generated. Then the image of $\Lambda|_C$ is isomorphic to $\Tc \otimes E$.
\end{proposition}
\begin{proof}

As $\varphi^* \cO_{\PP^{g-1}}(-1) \cong \Tc$, the pullback $\varphi^*\Lambda = \Lambda|_C$ is a map
\[ \Lambda|_C \colon \Tc \otimes H^0 (C, E) \to \Oc \otimes H^1 (C, E) . \]

Write $L := \det (E)$, a line bundle of degree $r(g-1)$. Then $\det (\Kc \otimes E^*) = \Kc^r \otimes L^{-1}$. As $\Kc \otimes E^*$ is globally generated, the evaluation sequence
\[ 0 \to \Kc^{-r} \otimes L \to \Oc \otimes H^0 (C, \Kc \otimes E^*) \to \Kc \otimes E^* \to 0 \]
is exact. For each $p \in C$, the image of $\left( \Kc^{-r} \otimes L \right)|_p$ in $H^0 (C, \Kc \otimes E^*)$ is exactly $\C \cdot t_p$, where $t_p$ is the unique section, up to scalar, of $\Kc \otimes E^*$ vanishing at $p$.

Dualising, we obtain a diagram
\[ \xymatrix{ 0 \ar[r] & \Tc \otimes E \ar[r] & \Oc \otimes H^0 (C, \Kc \otimes E^*)^* \ar[r]^-{e} & \Kc^r \otimes L^{-1} \ar[r] & 0 \\
 & \Tc \otimes H^0 (E) \ar[r]^{\Lambda|_C} \ar[ur] & \Oc \otimes H^1 (C, E) \ar[u]^\wr_{\text{Serre}} & & } \]
Here $e_p$ can be identified up to scalar with the map $f \mapsto f( t_p )$ where $t_p$ is as above. 

Now for each $p \in C$, the image
\[ \left[ \beta_p \right] \cup H^0 (C, E) \ \subset \ H^1 (C, E) \ \cong \ H^0 (C, \Kc \otimes E^*)^* \]
annihilates $t_p \in H^0 (C, \Kc \otimes E^*)$, since the principal part $\beta_p \cdot t_p$ is everywhere regular. Therefore, $\Lambda|_C$ factorises via $\Ker (e) = \Tc \otimes E$. Since $\rank (\Lambda|_C) \equiv r$ by Lemma \ref{rankr}, we have $\Image ( \Lambda|_C ) \cong \Tc \otimes E$. \end{proof}

\begin{remark} 
A straightforward computation shows also that 
$$
\Ker (\Lambda|_C) \cong \Tc \otimes L^{-1} \text{ and } \Coker (\Lambda|_C) \cong \Kc^r \otimes L^{-1} .
$$ 
\end{remark}

We will also want to study the transpose $\Lambda^t$, which we will consider as a map
\[ \Lambda^t \colon \cO_{\PP}(-1) \otimes H^0(C, \Kc \otimes E^*) \ \to \ \cO_{\PP} \otimes H^1 (C, \Kc \otimes E^*) . \]
The proof of Proposition \ref{recovery} also shows:
\begin{corollary} \label{transpose} Let $E$ and $\Lambda$ be as above. Then the image of $\Lambda^t|_C$ is isomorphic to $E^*$. \end{corollary}

\begin{remark} In order to describe the cokernel of $\Lambda|_C$, it is also enough to know in which points of $C$ a row of $\Lambda|_C$ vanishes. Dualising the sequence
\[ 0 \to \Kc^r \otimes L^{-1} \to \Oc \otimes H^0 (C, \Kc \otimes E^*) \to \Kc \otimes E^* \to 0, \]
we see that $H^0 (C, \Kc \otimes E^*)^*$ is canonically identified with a subspace of $H^0(C,\Kc^r \otimes L^{-1})$. Using the description of
\[ \Tc \otimes H^0 (C, E) \ \stackrel{\Lambda|_C}{\longrightarrow} \ \Oc \otimes H^1 (C, E) \ \stackrel{\sim}{\longrightarrow} \ \Oc \otimes H^0 (C,\Kc\otimes E^*)^* \]
as in the above proof, one sees that a row vanishes exactly in a divisor associated to $\Kc^r \otimes L^{-1}$. Hence, the cokernel is isomorphic to $\Kc^r \otimes L^{-1}$. \end{remark}

\subsection{Uniqueness of the determinantal representation of the tangent cone}

In order to show the desired uniqueness of the determinantal representation of the tangent cone, we use a classical result of Frobenius. See \cite{Fro} and also for a modern proof \cite{Dieu49}, \cite{W87} and the references there. For the sake of completeness we will give a sketch of a proof following Frobenius.

\begin{proposition}\label{complexes} 
Suppose $r \ge 1$. Let $A$ and $B$ be $(r+1)\times (r+1)$ matrices of independent linear forms, such that the entries of $A$ are linear combinations of the entries of $B$ and $\det(A)=c\cdot\det(B)$ for a nonzero constant $k \in \C$. Then, there exist invertible matrices $S,T\in Gl(r+1,\C)$, unique up to scalar, such that $A=S \cdot B\cdot T$ or $A=S \cdot B^t\cdot T$.
\end{proposition}

\begin{proof}[Proof by Frobenius {\cite[pages 1011-1013]{Fro}}] 
 Note that for $r\geq 1$ only one of the above cases can occur and the matrices $S$ and $T$ are unique up to scalar. Indeed, let $A=SBT=S'BT'$ and set $b_{ii}=1$ and $b_{ij}=0$ if $i\neq j$, then $ST=S'T'$. Set $U=T(T'^{-1})=S(S'^{-1})$, thus $UB=BU$. Since $U$ commutes with every matrix, we have $U=k\cdot E_r$ and hence 
 $S'=k\cdot S$ and $T'=\frac{1}{k}\cdot T$. Similar one can show that $B^t$ is not equivalent to $B$. 
 Note also that there is no relation between any minors of $A$ or $B$.
 
 For $l=0,\dots, r$, let $c^{l}_{ij}$ be the coefficient of $b_{ll}$ in $a_{ij}$ and let $y$ be a new variable. We substitute $b_{ll}$ with $b_{ll}+y$ in $A$ and $B$ and get new matrices, denoted by $\left(a_{ij}+y\cdot c^{l}_{ij}\right)$ and $B^l$, respectively. Since $\det B^l$ is linear in $y$, the coefficient of $y^2$ in $\det \left(a_{ij}+y\cdot c^{l}_{ij}\right)=\det B^l$ has to vanish. But the coefficient is the sum of products of $2\times  2$ minors of $\left(c^{l}_{ij}\right)$ and $(r-1)\times (r-1)$ minors of $A$. Since there are no relations between any minors of $A$, all $2\times  2$ minors of $\left(c^{l}_{ij}\right)$ vanish. Hence, $\left(c^{l}_{ij}\right)$ has rank one for any $l$ and we can write $c^{l}_{ij}=p^{l}_iq^{l}_j$ where $p^{l}$ and $q^{l}$ are column and row vectors, respectively. 
 
 Let $B_0=B|_{\{b_{ij}=0, i\neq j\}}$ and $A_0=A|_{\{b_{ij}=0, i\neq j\}}$. Then 
 $$
 A_0=PB_0 Q
 $$
 where $P=\left(p^{l}_i\right)_{0\leq i,l\leq r}$ and $Q=\left(q^{l}_j\right)_{0\leq l,j\leq r}$. Since $\det(A_0) = c\cdot \det (B_0) = c\cdot b_{00}\cdot ... \cdot b_{rr}$, we get $\det(P)\cdot \det(Q)=c$, hence $P$ and $Q$ are invertible.  
 
 Let $\widetilde{B}=P^{-1}AQ^{-1}$. By definition $\widetilde{B}|_{\{b_{ij}=0, i\neq j\}}=B_0$. Thus, the entries $\widetilde{b_{ij}}$ for $i\neq j$ and $v_i=\widetilde{b_{ii}}-b_{ii}$ are linear function in $b_{ij}$ for $i\neq j$. Furthermore, we have 
 $$\det(\widetilde{B})= \det(P^{-1}AQ^{-1}) = \det(P^{-1}Q^{-1})\cdot \det(A) = \frac{1}{c}\cdot c\cdot \det(B)= \det(B).$$
 Comparing the coefficient of $b_{11}b_{22}\cdots b_{rr}$ in $\det(\widetilde{B})$ and $\det(B)$, we get $v_0=0$. Similarly, $v_i=0$ for $0\leq i \leq r$. Comparing the coefficients of $b_{22}\cdots b_{rr}$, we get $b_{12}b_{21}=\widetilde{b_{12}}\widetilde{b_{21}}$ and in general 
 $$
 b_{ij}b_{ji}=\widetilde{b_{ij}}\widetilde{b_{ji}},\ i\neq j. 
 $$
 Comparing the coefficients of $b_{33}\cdots b_{rr}$, we get $b_{12}b_{23}b_{31}+b_{21}b_{13}b_{32}=\widetilde{b_{12}}\widetilde{b_{23}}\widetilde{b_{31}}+\widetilde{b_{21}}\widetilde{b_{13}}\widetilde{b_{32}}$ and in general 
 $$
 b_{ij}b_{jk}b_{ki}+b_{ji}b_{ik}b_{kj}=\widetilde{b_{ij}}\widetilde{b_{jk}}\widetilde{b_{ki}}+\widetilde{b_{ji}}\widetilde{b_{ik}}\widetilde{b_{kj}},\ i\neq j\neq k\neq i.
 $$ 
 A careful study of these equations shows that either 
 $$
 \widetilde{b_{ij}}=\frac{k_i}{k_j} b_{ij} \text{ and } \widetilde{B}=K B K^{-1} 
 \text{ or }
  \widetilde{b_{ij}}=\frac{k_i}{k_j} b_{ji} \text{ and } \widetilde{B}=K B^t K^{-1}
 $$
 where $K=\left(k_i \delta_{ij}\right)_{0\le i,j \le r}$. The claim follows. 
\end{proof}

We now assume that $E$ is a Petri trace injective bundle. Let $\Lambda=\left(l_{ij}\right)$ be a determinantal representation of the tangent cone $\cT_{\Oc}(\Theta_E)$ as above. By Petri trace injectivity, the matrix $\Lambda$ is $(r+1)$-generic (see \cite{Eis88} for a definition), that is, there are no relations between the entries $l_{ij}$ or any subminors of $\Lambda$. 

\begin{corollary}
\label{uniqueDetRepresentation}
 For a curve of genus $g\ge (r+1)^2$ and a Petri trace injective bundle $E$ with $r+1$ global sections of degree $r(g-1)$, any determinantal representation of the tangent cone $\cT_{\Oc}(\Theta_E)\subset |\Kc|^*$ is equivalent to $\Lambda$ or $\Lambda^t$.
\end{corollary}

\begin{proof}
 Let $\alpha$ be any determinantal representation of the tangent cone $\cT_{\Oc}(\Theta_E)$ in $|\Kc|^*$. Then, $\alpha$ is an $(r+1)\times (r+1)$ matrix of linear entries, since the degree of the tangent cone is $r+1$. Furthermore, the entries $\alpha_{ij}$ of $\alpha$ are linear combinations of the entries $l_{ij}$ of $\Lambda$. Indeed, assume for some $k,l$ that $\alpha_{kl}$ is not a linear combination of the $l_{ij}$. Then, $\cT_{\Oc}(\Theta_E)$ would be the cone over $V(\alpha_{kl})\cap \cT_{\Oc}(\Theta_E)$. Hence, the vertex of $\cT_{\Oc}(\Theta_E)$ defined by the entries $l_{ij}$ would have codimension strictly less than $(r+1)^2$; a contradiction to the independence of the $l_{ij}$. The corollary follows from Proposition \ref{complexes}.
\end{proof}

\section{Injectivity of the theta map}

\begin{theorem} Suppose $r \ge 3$. Let $C$ be a Petri general curve of genus $g \ge (2r+2)(2r+1)$. Then the theta map $\cD \colon \SU_C(r) \dashrightarrow |r \Theta|$ is generically injective. \label{Thetageninj} \end{theorem}

\begin{proof}

Let $V \in \SU_C (r)$ be a general stable bundle. By Theorem \ref{MExists}, there exists $M \in \Theta_V$ such that $h^0 (C, V \otimes M) = r+1$, the bundle $V \otimes M =: E$ is Petri trace injective, and $E$ and $\Kc \otimes E^*$ are globally generated.

Note that tensor product by $M^{-1}$ defines an isomorphism $\Theta_V \isom \Theta_E$ inducing an isomorphism $\cT_{M} ( \Theta_V ) \isom \cT_{\Oc} ( \Theta_E )$. In order to use the results of the previous sections, we will work with $\Theta_E$.

Now let
$$
\alpha \colon \cO_{\PP^{g-1}}(-1) \otimes \C^{r+1} \to \cO_{\PP^{g-1}} \otimes \C^{r+1}
$$ 
be a map of bundles of rank $r+1$ over $\PP^{g-1}$ whose determinant defines the tangent cone $\cT_{\Oc} (\Theta_E)$. 
 By Corollary \ref{uniqueDetRepresentation}, the map $\alpha$ is equivalent either to $\Lambda$ or $\Lambda^t$, where $\Lambda$ is the representation given by the cup product mapping as defined in \S \ref{reconstructionTangentCone}. Therefore, by Proposition \ref{recovery} and Corollary \ref{transpose}, the image $E'$ of $\alpha|_C$ is isomorphic either to $\Tc \otimes E = V \otimes M \otimes \Tc$ or to $E^* = V^* \otimes M^{-1}$. Thus $V$ is isomorphic either to
\begin{equation} E' \otimes \Kc \otimes M^{-1} \quad \quad \hbox{ or to } \quad \quad (E')^* \otimes M^{-1} . \label{twoposs} \end{equation}
Now since in particular $g > (r+1)^2$, the open subset $\{ M \in \Pic^{g-1}(C) : h^0 (C, V \otimes M) = r+1 \} \subseteq B^{r+1}_{1, g-1}(V)$ has a component of dimension $g - (r+1)^2 \ge 1$. Therefore, we may assume that $M^{2r} \not \cong \Kc^r$. Thus only one of the bundles in (\ref{twoposs}) can have trivial determinant. Hence there is only one possibility for $V$.

In summary, the data of the tangent cone $\cT_M ( \Theta_V )$ and the point $M$, together with the property $\det (V) = \Oc$, determine the bundle $V$ up to isomorphism. In particular, $\Theta_V$ determines $V$. \end{proof}

\begin{remark} The involution $M \mapsto \Kc \otimes M^{-1}$ defines an isomorphism of varieties $\Theta_V \isom \Theta_{V^*}$. We observe that the transposed map $\Lambda^t$ occurs naturally as the cup product map defining the tangent cone $\cT_{\Kc \otimes M^{-1}} \left( \Theta_{V^*} \right)$. \end{remark}

\begin{remark} \label{failure} If $C$ is hyperelliptic, then the canonical map factorises via the hyperelliptic involution $\iota$. Thus the construction in \S \ref{recovery} can never give bundles over $C$ which are not $\iota$-invariant. We note that Beauville \cite{Bea88} showed that in rank $2$, if $C$ is hyperelliptic then the bundles $V$ and $\iota^* V$ have the same theta divisor. \end{remark}

\end{document}